\documentclass[a4paper, 11pt]{article}

\usepackage[a4paper, left=3.5cm,top=3.4cm,right=3.5cm,bottom=3.8cm]{geometry}
\usepackage{hyperref}

\usepackage{graphicx}
\usepackage{epstopdf}
\usepackage{bold-extra}
\usepackage{mathtools}
\usepackage{psfrag}
\usepackage{graphicx}
\usepackage{pgfplots}
\usepackage{tikz}
\usepackage{tikz-3dplot}
\usepackage{amssymb}
\usepackage{subcaption}
\usepackage{enumitem}
\setitemize{label=\scriptsize{$\blacksquare$},topsep=0em,itemsep=0em}
\setenumerate{label=(\alph*),topsep=0em,itemsep=0em}
\usepackage{framed}

\usepackage{amssymb}

\usetikzlibrary{calc}

\usepackage{authblk}
\usepackage{xspace}

\usepackage{enumitem}
\setitemize{label=\scriptsize{$\blacksquare$}}
\setenumerate{label=(\alph*)}
\usepackage{amssymb,amsthm}

\usepackage{soul}
\usepackage{xcolor}
\colorlet{lred}{red!40}
\colorlet{lgreen}{green!40}
\colorlet{lblue}{blue!40}

\newtheorem{theorem}{Theorem}
\newtheorem{alg}[theorem]{Algorithm}
\newtheorem{lemma}[theorem]{Lemma}

\newtheorem{proposition}[theorem]{Proposition}

\newtheorem{corollary}[theorem]{Corollary}

\numberwithin{equation}{section}
\numberwithin{figure}{section}
\numberwithin{theorem}{section}

\newcommand{\R}{\mathbb R}

\newcommand{\N}{\mathbb N}

\newcommand{\edot}{\,\cdot\,}
\newcommand{\trans}{\mathsf{T}}

\newcommand{\f}{\mathbf}

\newcommand{\Wo}{\mathbf{W}}
\newcommand{\Ao}{\mathbf{A}}

\newcommand{\Rad}{\mathbf{R}}

\newcommand{\Ft}{\mathbf{F}}

\newcommand\abs[1]{\left\vert#1\right\vert}

\newcommand\norm[1]{\left\Vert#1\right\Vert}

\newcommand\set[1]{\bigl\{#1\bigr\}}
\newcommand{\kl}[1]{\left(#1\right)}

\newcommand{\skl}[1]{(#1)}

\makeatletter
\newcommand*\bigcdot{\mathpalette\bigcdot@{.6}}
\newcommand*\bigcdot@[2]{\mathbin{\vcenter{\hbox{\scalebox{#2}{$\m@th#1\bullet$}}}}}
\makeatother

\newcommand{\AxisRotator}[1][rotate=0]{%
	\tikz [x=0.25cm,y=0.60cm,line width=.2ex,-stealth,#1] \draw (0,0) arc (-150:150:1 and 1);%
}

\allowdisplaybreaks

\title{Full field inversion in photoacoustic tomography with variable sound speed}

\author{Gerhard Zangerl}
\author{Markus Haltmeier}
\affil{Department of Mathematics, University of Innsbruck\\
Technikestra{\ss}e 13, 6020 Innsbruck, Austria\\
E-mail: {\tt \{gerhard.zangerl,markus.haltmeier\}@uibk.ac.at}}

\author{Linh V. Nguyen}
\affil{Department of Mathematics, University of Idaho\\
875 Perimeter Dr, Moscow, ID 83844, US\\
E-mail: {\tt lnguyen@uidaho.edu}}

\author{Robert Nuster}
\affil{Department of Physics, University of Graz\\
Universit\"atsplatz 5, 8010 Graz, Austria\\
E-mail: {\tt ro.nuster@uni-graz.at}}

\date{August 2, 2018}

\begin{document}
\maketitle

\begin{abstract}
Recently, a novel measurement setup has been introduced to photoacoustic tomography, that
collects data in the form of projections of the full  3D acoustic pressure distribution
at a certain time instant. Existing imaging algorithms for this kind of data assume
a constant speed of sound. This assumption is not always met in practice and thus
leads to erroneous reconstructions.
In this paper, we present a two-step reconstruction method for full field
 detection photoacoustic tomography  that  takes variable speed of sound into account.
In the first step,  by applying   the inverse Radon transform, the pressure distribution at the measurement time  is reconstructed  point-wise from the projection data.
In the second step, one solves a final time wave inversion problem where
the initial pressure distribution is recovered from the known pressure distribution  at the measurement time. For the latter problem, we derive  an iterative solution approach, compute the required  adjoint operator,
and  show  its uniqueness and stability.
\end{abstract}

\section{Introduction} \label{sec:intro}

Photoacoustic tomography (PAT) is a hybrid imaging modality that combines high spatial resolution of ultrasound and high contrast of optical tomography  \cite{Bea11,KruKopAisReiKruKis00,WanAna11,Wan09b,XuWan06}. In PAT, a semitransparent sample is illuminated by a short laser pulse. As a result, parts of the optical energy are absorbed inside the sample. This
 causes an initial pressure distribution and a subsequent  acoustic pressure wave. The pressure wave is detected outside the
 investigated object and used to recover an image of the interior.

In standard  PAT, the induced waves are measured  on a surface  enclosing the  investigated object.
In the case of constant sound speed  and when the observation surface  exhibits a special geometry
(planar, cylindrical, spherical), initial pressure distribution can be recovered by closed-form inversion
formulas; see       \cite{AgrKucKun09,FilKunSey14,FinHalRak07,FinPatRak04,HalSchuSch05,Hal14,Hal14b,HalPer15b,Kow14,kuchment2014radon,Kun07a,Kun07b,KucKun08,XuWan05,XuWan06} and references therein. All these  algorithms assume that the acoustic pressure is known point-wise on a detection surface. Due to finite width of the commonly used piezoelectric elements this assumption is only approximately satisfied. Therefore,  the concept of integrating detectors has been invented as an alternative approach to PAT. Integrating detectors measure the integral of acoustic pressure over planes, lines or circles. Closed-form inversion formulas that incorporate integrated pressure data have been developed in \cite{BurBauGruHalPal07,haltmeier2004,paltauf2007,zangerl2009exact}.

Inspired by the concept of integrating  line detectors, a full field detection method that is capable to image the
whole  acoustic field around an object has been invented in \cite{nuster2010full, NusSle14}. In full field detection
PAT (FFD-PAT), a phase contrast method is used to obtain data in the form of
2D projections  of the pressure field at a time instant $T$.  If  the measurement time $T$
is sufficiently large, then the acoustic pressure has
essentially left the object.  As shown in \cite{nuster2010full,NusSle14}, in the case of constant sound speed,
projection data from different directions allow for a full 3D reconstruction of the initial pressure by
Radon or Fourier transform techniques.

Existing image reconstruction methods for FFD-PAT assume a constant speed of sound.
However, there are relevant cases when the assumption of constant speed of sound is inaccurate \cite{JinWan06, KuFoXu05}.
For example, it is known that acoustic properties vary within female human breasts. Consequently, for accurate image reconstruction, variable speed of sound  has to be incorporated in the wave propagation model.
Iterative methods are capable to deal with this assumption.
In the case of standard PAT, such methods have been studied in \cite{arridge2016adjoint,belhachmi2016direct,HalLinh17,huang2013full,nguyen18a,SteUhl09}. Therein the spatially variable speed of sound is assumed to be a smooth function and bounded from below.
Moreover, it is assumed to satisfy the so called nontrapping condition, which means that the supremum of the lengths of all geodesics connecting any two points inside the volume
enclosed by the measurement surface $S$ is finite. Under this assumption, it is known that the initial  pressure can be stably reconstructed from
pressure data restricted to $S \times [0,T]$ provided that the measurement time $T$ is sufficiently large.

In this paper, we study image reconstruction in FFD-PAT with a spatially variable speed of sound. We will give a precise mathematical formulation of FFD-PAT
and describe the inverse problem we are dealing with (see Section \ref{sec:model}).
For its solution we propose a two-step process. In the first step, the acoustic pressure at time $T$ is reconstructed pointwise from the full field data. In the second step, we recover the desired initial pressure from the pressure known
for fixed measurement time $T$.
The first step can be approximated by inverting the well-known Radon transform. The second step consists in a final time
wave inversion problem  with spatially varying speed of sound. To the best of our knowledge, the latter has not been addressed in the literature so far.
For its solution, we develop iterative reconstruction methods based on an explicit  computation of an adjoint problem. As main theoretical results, we establish uniqueness and stability of the final time wave inversion problem. In particular, this implies  linear convergence for the proposed iterative
reconstruction methods.

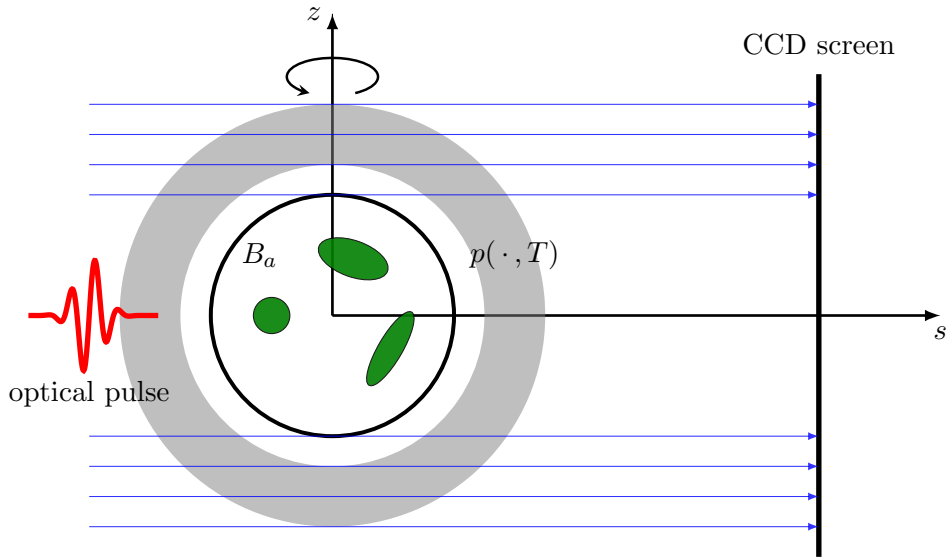
\begin{figure}[htb!]
	\centering
	\begin{tikzpicture}[scale=0.8]
	\draw[thick,->, black,-latex, line width = 1pt ] (0,0) -- (10,0) node[anchor=north]{$s$};
	\draw[thick,->,-latex, line width = 1pt] (0,0) -- (0,5)         node[anchor=east]{$z$};
	\draw (0,0) -- (0,5)         node at (0,4)  {\AxisRotator[rotate = 90 ]};
	\draw[line width = 2pt  ] (8,4) -- (8,-4)  ;
	\draw[line width = 1.5pt] (0, 0) circle (2);
	\draw[rotate = -20, fill = green!50!black, opacity = 0.9] (0,1.0) ellipse (0.6cm and 0.3cm);
	\draw[rotate = 0, fill = green!50!black, opacity = 0.9] (-1.0,0) ellipse (0.3cm and 0.3cm);
	\draw[rotate = -120, fill = green!50!black, opacity = 0.9] (0,1.1) ellipse (0.7cm and 0.2cm);
	\path[fill = gray, opacity = 0.5, even odd rule] (0,0) circle (2.5) -- (0,0) circle (3.5);
	\draw[->, -latex, color = blue, opacity = 0.8] (-4, 3.5)--(8, 3.5);
	\draw[->, -latex, color = blue, opacity = 0.8] (-4, 3)--(8, 3);
	\draw[->, -latex, color = blue, opacity = 0.8] (-4, 2.5)--(8, 2.5);
	\draw[->, -latex, color = blue, opacity = 0.8] (-4, 2)--(8, 2);
	\draw[->, -latex, color = blue, opacity = 0.8] (-4, -2)--(8, -2);
	\draw[->, -latex, color = blue, opacity = 0.8] (-4, -2.5)--(8, -2.5);
	\draw[->, -latex, color = blue, opacity = 0.8] (-4, -3)--(8, -3);
	\draw[->, -latex, color = blue, opacity = 0.8] (-4, -3.5)--(8, -3.5);
	\node at (8,4.5) {\text{CCD  screen} };
	\node at (3, 1) {$p(\edot,T)$ };
	\node at (-4,-1.3) {\text{optical pulse} };
	\node at (-1.2,1.0) {$B_a$};
	\draw[thick, red, line width = 2pt] plot[domain=-5*pi:-9, samples=250]  (\x/pi,{sin(5* \x r) * exp( -(\x + 4*pi)^2)});
	\end{tikzpicture}\caption{\textbf{Illustration of FFD-PAT with variable speed of sound.}
	An object is illuminated by a short pulse of electromagnetic radiation at $t=0$.
	After a sufficiently large time $T > 0$,  the acoustic pressure $p(\edot, T)$ has almost left the
	investigated object and linear projections of $p(\edot, T)$  along
	lines not intersecting $B_a$ and perpendicular to the CCD screen  are recorded.
	After that, the object is rotated around the $(0,0,1)$ axis
	and the measurement process is repeated.}\label{fig:FFD_setup}
\end{figure}

\section{Full field detection photoacoustic tomography}\label{sec:model}

In this section, we describe a mathematical model for FFD-PAT including variable
sound speed case, and state the inverse problem under consideration.
Additionally, we outline the proposed two-step reconstruction procedure
and formulate the final time inverse problem.

\subsection{Mathematical model}

In the case of variable sound speed, acoustic wave propagation in PAT is commonly described by the widely
accepted model~\cite{HalLinh17,SteUhl09,huang2013full, JinWan06}
\begin{align}\label{eq:wave1}
 p_{tt} (\f x, t ) - c^2(\f x) \Delta p(\f x, t) &= 0,  \, && \, \, (\f x, t) \in \R^3 \times \R_{> 0} \\ \label{eq:wave2}
p(\f x, 0 ) &= f(\f x),  \, \,          && \, \,    \f x  \in \R^3   \\ \label{eq:wave3}
p_t(\f x, 0)   &= 0,  \,       &&  \, \,  \f x  \in \R^3.
\end{align}
Here $c(\f x) > 0$ is the sound speed at location $\f x \in \R^3$, and $f \in C_0^{\infty}(\R^3 )$ is the  initial pressure distribution
that encodes the inner structure of the object. Throughout this text it is assumed that the object is contained inside
$B_a = \set{\f x \in \R^3 \mid \norm{\f x} <a}$,
the ball of radius $a$ centered at the origin and that the sound speed is smooth, positive and has the constant value $c_0$ outside $B_a$.

In FFD-PAT, linear projections (integrals along straight lines) of the 3D pressure field
$p(\edot, T)$  for a fixed time $T>0$ are  recorded; compare Figure \ref{fig:FFD_setup}.
This can be  implemented using a  special phase contrast  method  and a CCD-camera that
recorders full field projections of the pressure field \cite{nuster2010full,NusSle14}.
The projections are collected for rotation angles $\alpha \in [0, \pi]$
around the $e_3 =(0,0,1)$ axis and are given by
\begin{multline}\label{eq:dataFFDPAT}
g_a(\alpha ,s, z)
=  \int_{\R} p (s \cos(\alpha) - t \sin(\alpha) , s \sin(\alpha) + t \cos(\alpha), z , T) \, dt
\\
\text{for } (\alpha , s, z)\in M_a \coloneqq \set{(\alpha, s,z) \in [0, \pi] \times \R^2  \mid s^2+z^2 \geq R^2} \,.
\end{multline}
Here  $M_a $ determines the set of admissible projections, where the defining condition
$s^2 + z^2 \geq R^2$ means that in practice only pressure integrals over those lines are recorded,
which  do not intersect  the possible support of the imaged object.

\begin{figure}[htb!]
	\centering
	\tdplotsetmaincoords{70}{110}
	\begin{tikzpicture}[scale=4]
	\node at (-0.1,0.0) {$D_{a,z}$};
	\draw[fill = gray, opacity = 0.4] (0,0)  circle (0.2);
	\draw[fill = black] (0,0)  circle (0.01);
	\path[draw=none,fill=gray, fill opacity = 0.4,even odd rule]
	(0,0) circle (0.3)
	(0,0) circle (0.6);
	\draw[line width = 1.5pt] (0,0)  circle (0.2);
	\draw[blue, rotate = 15,  ] \foreach \y in {0.2, 0.3,...,0.7}
	{  (-0.8,\y)-- (0.8,\y) };
	\draw[blue, rotate = 15 ] \foreach \y in {-0.2,-0.3,...,-0.7}
	{(-0.8,\y)-- (0.8,\y) };
	\draw[blue, rotate = 15] \foreach \y in {-0.2,-0.3,...,-0.7}
	{(-0.8,\y)-- (0.8,\y) };
	\draw[blue, rotate = 15] \foreach \y in {0.2, 0.3,...,0.7}
	{  (-0.8,\y)-- (0.8,\y) };
	\draw[blue, rotate = -15] \foreach \y in {-0.2,-0.3,...,-0.7}
	{(-0.8,\y)-- (0.8,\y) };
	\draw[blue, rotate = -15] \foreach \y in {0.2, 0.3,...,0.7}
	{  (-0.8,\y)-- (0.8,\y) };
	\end{tikzpicture}\caption{By measuring full field  projections
	over lines  that do not intersect the ball $B_a$,
	for any plane $\R^2 \times \{z\}$, integrals of $p(\edot, T)$ over lines  all lines
	in $\R^2 \times \{z\}$ are measured that not intersect the disc
	$D_{a,z} \coloneqq  B_a \cap (\R^2 \times \{z\})$. For $\abs{z} > R$, this yields
	the exterior problem for the Radon transform.}\label{fig:ext_Problem}
\end{figure}
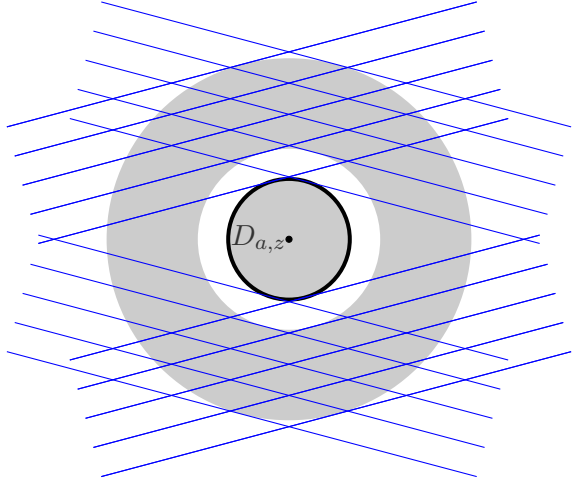

\subsection{Description of the inverse problem}

In order to describe the  inverse  problem of FFD-PAT in a more compact way we
introduce some further notation. First, we define the operator
\begin{equation}\label{eq:Wo}
\Ao \colon C_0^{\infty} (B_a ) \to C_0^{\infty} (\R^3) \colon
f \mapsto p(\edot,T)\,,
\end{equation}
where  $p$ denotes the solution of \eqref{eq:wave1}--\eqref{eq:wave3} with initial data $f$.
The operator  $\Ao$ maps the initial data $f$ to the solution (full field) of
the wave equation   \eqref{eq:wave1} at  the given measurement time $T>0$.
Second, we define the operator
\begin{equation}\label{eq:Radon}
\begin{aligned}
&\Rad \colon   C_0^\infty \left( \R^3 \right)  \rightarrow C^\infty_0\left([0, \pi] \times \R^2  \right)
\\
&(\Rad h)(\alpha, s, z ) \coloneqq
\int_{\R} h(s \cos(\alpha) - t \sin(\alpha) , s \sin(\alpha) + t \cos(\alpha), z ) \,dt
\end{aligned}
\end{equation}
Note that  for any fixed $z \in \R$, the function $(\Rad h)(\,\cdot\, , z )$ is the
Radon transform of $h (\,\cdot\,, z)$ in the horizontal plane $\R^2 \times \set{z}$.
Finally, we define the restricted Radon transform
\begin{equation}\label{eq:RadonR}
\Rad_a \colon   C_0^\infty \left(\R^3\right)  \rightarrow C^\infty\left(M_a \right)
\colon h \mapsto (\Rad  h)|_{M_a} \,,
\end{equation}
where $M_a$  is defined in \eqref{eq:dataFFDPAT}.
For $|z| \leq R$,   $(\Rad_a h)(\,\cdot\, , z )$  is the  exterior Radon transform  of
$h(\,\cdot\, , z )$ for lines not intersecting
$\set{(x,y) \in \R^2 \mid x^2+y^2 <  R^2-z^2}$; compare  Figure \ref{fig:ext_Problem}.
Otherwise,   $(\Rad_a h)(\,\cdot\, , z )$  coincides with the standard
Radon transform of $h(\,\cdot\, , z )$.

Using the operator notation introduced above we can  write the
inverse problem of  FFD-PAT in the  form
\begin{align}\label{eq:forward_ffd}
\text{Recover $f$ from data } \quad g_a =   \Rad_a  \Ao f  \,.
\end{align}
Evaluation of $ \Rad_a  \Ao f $ will be referred  to as the forward problem
in FFD-PAT. In this paper we study the  solution of the inverse
problem~\eqref{eq:forward_ffd}.

\subsection{Two stage reconstruction}\label{sec:2stage}

One possible approach to solve the inverse problem of    FFD-PAT
is to directly recover $f$ from  data in \eqref{eq:forward_ffd}  via iterative
methods. Typically, each iteration step will require the evaluation of $\Rad_a \Ao $ and
$( \Rad_a \Ao )^* = \Ao^* \Rad_a^*$. In this paper, we consider a two-step approach
where we first invert $ \Rad_a  $ via direct method
and then use an iterative method to  invert $\Ao$. This avoids  repeated and
time consuming evaluation of  $\Rad_a$ and its adjoint.

The proposed two stage reconstruction method  consists of the following:

\begin{itemize}
\item \textbf{Inverse Radon  transform:}
In this first reconstruction step, assume that projection data $g_a = \Rad_a \Ao f$ are given.
Assuming $T>0$ to be sufficiently large,   we  consider
the extension $g \colon [0, \pi] \times \R^2   \to \R$ by
$g(\alpha, x, z ) = g_a(\alpha, x, z )$ for $(\alpha, x, z ) \in M_a$
and $g(\alpha, x, z ) = 0$ otherwise.
We then define an  approximation  to $\Ao f$
by applying an inversion formula  of the  Radon
transform in  planes $\R^2 \times \{z\}$. 
Here we use the well-known filtered backprojection
formula (see \cite{Nat86}) which yields
\begin{equation*}
\Ao f  (x,y,z) \simeq  \Rad^\sharp g (x,y,z) \coloneqq  \frac{1}{2 \pi^2}
 \int_{0}^{\pi}  \mathrm{P.V.}\!\! \int_{\R} \frac{\kl{ \partial_s  g} \skl{\alpha, s, z} d s}{(x \cos(\alpha) +
 y \sin(\alpha))-s}
  d\alpha \,,
 \end{equation*}
where $\mathrm{P.V.}$ denotes the principal  value  integral.

\item  \textbf{Final time wave inversion:}
For the  second step we assume that an approximation
$h \simeq \Ao f$ to the 3D acoustic field at time $T$ is given.
This yields the final time wave inversion problem
\begin{equation} \label{eq:waveP}
\text{Recover $f$ from data } \quad  \Ao f = h.
\end{equation}
To the best of our knowledge, the problem has not been considered so far and its investigation
will be the main theoretical  focus of this work.
\end{itemize}

For solving the wave inversion problem  (second step), we
propose iterative solution  methods that are described in detail  in Section \ref{sec:MainResults}. Additionally, in Section~\ref{app:well} we derive
uniqueness and stability results  for \eqref{eq:waveP}.

Another option for solving the first step would be to work with the exterior Radon transform   \cite{Quin83,Quin88,kuchment2014radon}. However, we work with the standard Radon transform  after
replacing the missing values of $\Rad \Ao f$ by zero, since they are approximately zero for large enough $T$. Theoretically, the smallness is supported by the
following two facts. First,      in the case of non-trapping sound speed the
the known decay  estimate for the solution of
\eqref{eq:wave1} states that the following.

\begin{lemma}[Decay estimate \cite{Vain75}] \label{lem:decay}
Assume that the sound speed $c$ is non-trapping and the initial data  $f $ is supported in $B_a$.
 Then, for any $(k,m) \in \N^2$, the solution  $p$   of \eqref{eq:wave1}--\eqref{eq:wave3} satisfies
	\begin{equation}\label{ieq:decay}
	\left|  \frac{ \partial^{k+ |m|} p(\f x, t) }{\partial_t^k \partial_{\f x}^{|m|}   } \right| \leq C e^{-\delta t} \norm{f}_2 \quad
	\text{ for } (\f x,t)  \in B_a \times ( T, \infty) \,.
	\end{equation}
 Here $\delta > 0$ is a constant only depending  $c$  and $T$, and  $C$ is a constant
 depending  on the domain $B_a$.
 \end{lemma}

 Second, in the case of constant  sound speed, the  Radon transform $\Rad$ reduces the  initial value problem
 \eqref{eq:wave1} to a two dimensional wave equation with initial data $\Rad f$ which is
 supported in a disc of radius $a$. As the sound speed is  assumed to be constant
outside of $B_a$ in the constant sound speed  case, $\Rad \Ao f$  rapidly decays in
the complement of $M_a$.  For  non-trapping sound speed we numerically observed the same behaviour.
Theoretically investigating this issue, however,  is an open problem.

\section{Final time wave inversion problem}\label{sec:MainResults}

In this section  we study the final time wave inversion problem~\eqref{eq:waveP},
where  the forward operator $\Ao \colon f \mapsto p(\edot, T)$  is defined in \eqref{eq:Wo}.
According to standard results for the wave equation \cite{Treves2} the forward  operator extends to a bounded linear operator $\Ao \colon  L^2(B_a) \to L^2(\R^3)$. Below we establish uniqueness and stability results and derive an iterative reconstruction algorithm  using the CG method.

For constant sound speed, recovering the function $f$ from the solution at time $T$ of \eqref{eq:wave1}
with initial data $(0,f)$ instead of $(f,0)$  is equivalent to the the inversion from spherical means at
fixed radius. Uniqueness and in inversion method for this  problem has been obtained in the classical book
of Fritz John \cite{Joh82}. Neither for that case of initial data $(f, 0)$ nor in the variable sound speed case we are not aware of related results.

 \subsection{Uniqueness and stability}

The following theorem is the main theoretical result of this
paper and states that  the final time wave inversion problem~\eqref{eq:waveP}
has a unique solution that stably depends on the right-hand side.

\begin{theorem} \label{T:stab}
The operator $\Ao \colon  L^2(B_a) \to L^2(\R^3)$ is injective and bounded from below.
\end{theorem}

The proof  of Theorem \ref{T:stab} is presented in Appendix \ref{app:well}.
It states that
\begin{equation} \label{eq:b}
b \coloneqq  \inf \left\{ \frac{\norm{\Ao f}_{L^2(\R^3)}}{\norm{f}_{L^2(B_a)}} \; \bigg| \; f \in L^2(B_a) \right\} > 0 \,.
\end{equation}
In particular, $\Ao  \colon  L^2(B_a) \to R(\Ao) $ has a bounded
inverse, where  $R(\Ao)$ denotes the range of $\Ao$.
The latter result implies that standard  iterative methods for
 \eqref{eq:waveP}  converge linearly, similar to the case of standard PAT
 \cite{HalLinh17}.

\subsection{Solution by the CG method}

To find a solution of \eqref{eq:waveP} we use the conjugate gradient (CG) method  applied to the
normal equation $\Ao^* \Ao f =  \Ao^* h $.
The CG method has proven to be an accurate and fast reconstruction method for
the PAT  with variable sound speed~\cite{HalLinh17}.
Our numerical experiments confirm that the CG method is also efficient for  FFD-PAT, where it reads as
follows.

\begin{framed}
\begin{alg}[CG method for  \label{alg:cg} FFD-PAT]\mbox{}\vspace{-1em}
	\begin{enumerate}[label=(S\arabic*), itemsep=0em]
		\item  Initialize: $k = 0$, $r_0 = h - \Ao  f_0$, $d_0 = \Ao^* r_0$
		\item  While (not stop) do
		\begin{itemize}
      \item $\alpha_k = \norm{\Ao^* r_k}^2/ \norm{\Ao d_k}^2  $
		\item $  f_{k+1}  = f_k + \alpha_k d_k  $
		  \item $   r_{k+1}  = r_k - \alpha_k  \Ao d_k $
		  \item $  \beta_{k} = \norm{\Ao^* r_{k+1}}^2/\norm{\Ao^* r_k}^2  $
		  \item $  d_{k+1}   = \Ao^* r_{k+1} + \beta_k d_k$.
		\end{itemize}
		\end{enumerate}	
\end{alg}	
\end{framed}	

Using the injectivity and boundedness of $ \Ao $, Algorithm~\ref{alg:cg} generates a series of iterates
$f_k$ that converge to the unique solution  of the inverse source problem \eqref{eq:waveP}.
The stability of \eqref{eq:waveP} even implies that the CG method for FFD-PAT converges linearly.
More precisely, the sequence $(f_k)_{k \in \N}$ generated by Algorithm~\ref{alg:cg} satisfies
the estimate (see \cite{daniel1967conjugate})
\begin{equation*}
\forall k \in \N  \quad
\norm{f_{k}- f}_2 \leq 2\,  \frac{\norm{\Ao}}{b} \, \left( \frac{\norm{\Ao} -b}{\norm{\Ao}+b} \right)^k \,\norm{f_0 -f}_2^2 \,,
\end{equation*}
where $b$ is defined in \eqref{eq:b}.

 \subsection{The adjoint operator}

The CG method requires knowledge of the adjoint operator $\Ao^* \colon L^2(\R^3) \to  L^2(B_a) $ of $\Ao$.
We show that the adjoint operator  is again  determined by the solution of a wave equation.
More precisely, we have the following result:

\begin{theorem} \label{thm:dual}
Let $g \in C_0^\infty(\R^3)$, consider the  time reversed final state
problem for the wave equation,
\begin{alignat}{2}\label{eq:dual}
 q_{tt}(\f x, t)  - c(\f x)^2  \Delta q(\f x, t ) &= 0,   && \quad (\f x, t) \in \R^3 \times (-\infty,T)  \nonumber \\
 q(\f x, T )       &= g(\f x),      && \quad  \f x \in \R^3      \\
 q_t(\f x, T)    &= 0        	&& \quad  \f x  \in \R^3 \,,\nonumber
\end{alignat}
and let $\chi_{B_a} $ denote the indicator function of  $B_a$. Then,
\begin{equation}
\Ao^*  g =  \chi_{B_a}(\edot) q(\edot, 0)  \,.
\end{equation}
\end{theorem}

\begin{proof}
It is clearly sufficient to show $\Ao^*  g =  \chi_{B_a} u_t(\edot, 0) $, where
$u$ solves the wave equation $u_{tt}(\f x, t)  - c(\f x)^2  \Delta c(\f x, t ) = 0$ on $\R^3 \times (-\infty,T) $,
with the final state given by  $( u(\edot, T ), u_t(\edot, T ) )  = (0, f)$.
Using the weak formulation (similar to \cite{HalLinh17})  for the wave equation shows that for every
$v \in C_0^\infty(\R^3)$ we have
\begin{equation*}
	\int_0^T  \int_{\R^3} \frac{1}{c^{2}(\f x)} u_{tt}(\f x, t)  v(\f x, t)  \, d\f x dt  + \int_0^T \int_{\R^3} \nabla  u(\f x, t)
\cdot \nabla v \left(\f x , t \right) \, d \f x dt  = 0 \,.
	\end{equation*}
Two times integration by parts, rearranging terms and using the final state conditions
 $( u(\edot, T ), u_t(\edot, T ) )  = (0, f)$ yields
\begin{multline*}
\int_{\R^3} \frac{1}{c^{2}(\f x)}  \left[  f(\f x) v(\f x, T)
  - u_t(\f x, 0) v (\f x, 0)   + u (\f x, 0) v_t(\f x, 0) \right] \, d\f x
\\ =
\int _{0}^{T}  \int_{\R^3} u(\f x, t) \left[ c^{-2}(\f x)   v_{tt}(\f x, t) - \Delta v(\f x, t)  \right] d \f x dt .
\end{multline*}
By taking  $v $ as the solution of \eqref{eq:wave1}--\eqref{eq:wave3}  this yields
\begin{equation*}
\int_{\R^3}  \frac{1}{c^{2}(\f x)} g(\f x) \,  \Ao(f)(\f x)  d \f x =  \int_{\R^3}  \frac{1}{c^{2}(\f x)} u_t(\f x, 0) f(\f x) d\f x.
\end{equation*}
This implies $\Ao^*  g = \chi_{B_a} u_t (\edot, 0)  = \chi_{B_a} q(\edot, 0)$ and completes the proof.
\end{proof}

We can reformulate the  adjoint operator as follows
\begin{corollary} \label{cor:dual}
For  $g \in C_0^\infty(\R^3)$, let $q$ be the solution of
\begin{alignat}{2}\label{eq:dual2}
 q_{tt}(\f x, t)  - c(\f x)^2  \Delta q(\f x, t ) &= 0,   && \quad (\f x, t) \in \R^3 \times (0,\infty)  \nonumber \\
 q(\f x, 0 )       &= g(\f x),      && \quad  \f x \in \R^3      \\
 q_t(\f x, 0)     &= 0        	&& \quad  \f x  \in \R^3 \,.\nonumber
\end{alignat}
Then $\Ao^*  g =  \chi_{B_a}(\edot) q(\edot, T)$.
\end{corollary}

\begin{proof}
Clearly  $q$ solves \eqref{eq:dual2} if and only if $(x,t) \mapsto  q(x, T-t )$ solves  \eqref{eq:dual}.
Therefore the  claim follows from Theorem \ref{thm:dual}.
\end{proof}

\section{Numerical experiments}\label{sec:num}

In this section we present details on the implementation of
CG method (Algorithm~\ref{alg:cg}) for FFD-PAT, where the forward operator $\Ao$  and its adjoint
$\Ao^*$  given by the solution of \eqref{eq:waveP} and \eqref{eq:dual}, respectively.
Numerical experiments are conducted for two variable and two
trapping speed of sound models.

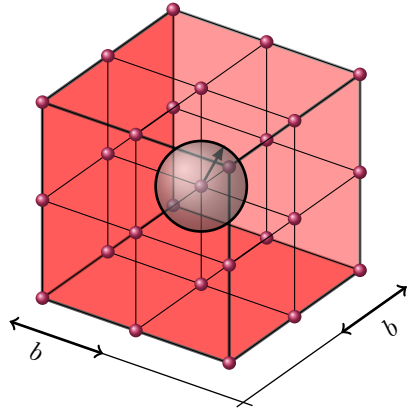
\begin{figure}[h!]
	\centering
	\tdplotsetmaincoords{60}{125}
	\tdplotsetrotatedcoords{0}{0}{0} 
	\begin{tikzpicture}[scale=1.5,tdplot_rotated_coords,
	cube/.style={very thick,black},
	grid/.style={very thin,gray},
	axis/.style={->,blue,ultra thick},
	rotated axis/.style={->,purple,ultra thick}]

	\draw[cube,fill=red, opacity=0.4] (0,0,0) -- (0,2,0) -- (2,2,0) -- (2,0,0) -- cycle;
	\draw[cube,fill=red, opacity=0.4] (0,0,0) -- (2,0,0) -- (2,0,2) -- (0,0,2) -- cycle;
	\draw[cube,fill=red, opacity=0.4] (0,0,2) -- (0,2,2) -- (2,2,2) -- (2,0,2) -- cycle;
	\draw[cube,fill=red, opacity=0.4] (2,2,2) -- (2,2,0) -- (0,2,0) -- (0,2,2) -- cycle;
	\draw[cube,fill=red, opacity=0.4] (2,0,0) -- (2,2,0) -- (2,2,2) -- (2,0,2) -- cycle;
	\node at (0.8,2.8,0)   { \rotatebox{40}{ \fontsize{10}{10}  $ b$} };
	\node at (2.8, 0.4,0)  { \rotatebox{-20}{ \fontsize{10}{10} $ b $} };

	\draw[->,-latex ,line width = 1pt]  (1,1, 1)-- (50:0.7cm) ;

	
	\draw[<-> ,  line width = 1pt]  (0,2.5,0)--(1,2.5,0);
	\draw[<-> , line width = 1pt]   (2.5,0,0)--(2.5,1,0);
	\draw[ line width = 0.5pt]      (2.5,0,0)--(2.5,2.6,0);
	\draw[line width = 0.5pt]       (0,2.5,0)--(2.6,2.5,0);

\foreach \x in {0,1,2}
\foreach \y in {0,1,2}
\foreach \z in {0,1,2}{
		\ifthenelse{  \lengthtest{\x pt < 2pt}  }
		{
			\draw [black]   (\x,\y,\z) -- (\x+1,\y,\z);
		}
		{
		}
		\ifthenelse{  \lengthtest{\y pt < 2pt}  }
		{
			\draw [black]   (\x,\y,\z) -- (\x,\y+1,\z);
		}
		{
		}
		\ifthenelse{  \lengthtest{\z pt < 2pt}  }
		{
			\draw [black]   (\x,\y,\z) -- (\x,\y,\z+1);
		}
		{
		}
		\shade[rotated axis,ball color = purple!80] (\x,\y,\z) circle (0.06cm);
	}
	\shade[ball color = gray!60, opacity = 0.6, line width = 1pt] (1,1,1) circle (0.4cm);
	\draw[line width = 1pt]  (1,1,1) circle  (0.4cm);
	\end{tikzpicture}\caption{The ball $B_a$ is considered to be contained inside a discrete 3D cubic region. The side length $l$ of the cube is chosen to be larger  than $4R$ in order to contain the full field $p(\edot, T)$.} \label{fig:3Dgrid}
\end{figure}

\subsection{Discretization and data simulation}

To implement the CG method, we have to discretize $\Ao$  and its adjoint
$\Ao^*$. For that  purpose we solve the forward and adjoint wave equation
\eqref{eq:waveP} and \eqref{eq:dual} on a cubical grid with nodes
\begin{align*}
 \f x_{i_1,i_2,i_3} \coloneqq -(b,b,b)  +\frac{2b}{N} (i_1,i_2,i_3)
 \quad \text{ for }  (i_1,i_2,i_3)   \in \set{0, \dots, N-1}  \,.
\end{align*}
For the solution \eqref{eq:waveP} and \eqref{eq:dual}  we use the $k$-space method \cite{cox2007,mast02}, which
we briefly explain it in Appendix \ref{app:kspace}.  The implementation of the $k$-space method
yields a $2b$ periodic solution.  The parameter $b$  is  selected sufficiently
large such that  the solution of  the wave equation \eqref{eq:wave1} with initial data supported in $B_a$
 coincides with its $2b$-periodic extension for all times $t \in [0, 2 T]$; compare Figure \ref{fig:3Dgrid}.

We denote by $X_N \subseteq \R^{N \times N \times N}$  the set of all
$\f f$ with $\f f_{i_1,i_2,i_3} = 0$ for  $\f x_ {i_1,i_2,i_3}   \not\in B_a$.
The discretized versions of $\Ao$  and its adjoint $\Ao^* $ are defined  by
\begin{align*}
&\Ao_{N,M} \colon X_N
\to   \R^{N \times N \times N} \colon  \f f  \mapsto \Wo_{N, M} \f f (\edot, M) \\
&\Ao_{N,M} ^\trans \colon \R^{N \times N \times N}  \to
X_N \colon  \f g
\mapsto \chi_{B_N} \Wo_{N, M} \f g (\edot, M) \,.
\end{align*}
Here $\Wo_{N, M}  \colon \R^{N\times N\times N} \to \R^{N\times N\times N\times (M+1)}$
denotes the discretized wave  propagation defined  by the $k$-space method
using the discrete time steps $j  T/M$ for $0 \leq j \leq M$.

The discrete and inverse Radon transforms $\Rad$ and $\Rad^\sharp$ are computed by the standard \textsc{Matlab} implementation of the Radon transform and its inverse.

\begin{figure}[htb!]
	\centering
	\includegraphics[width=\textwidth]{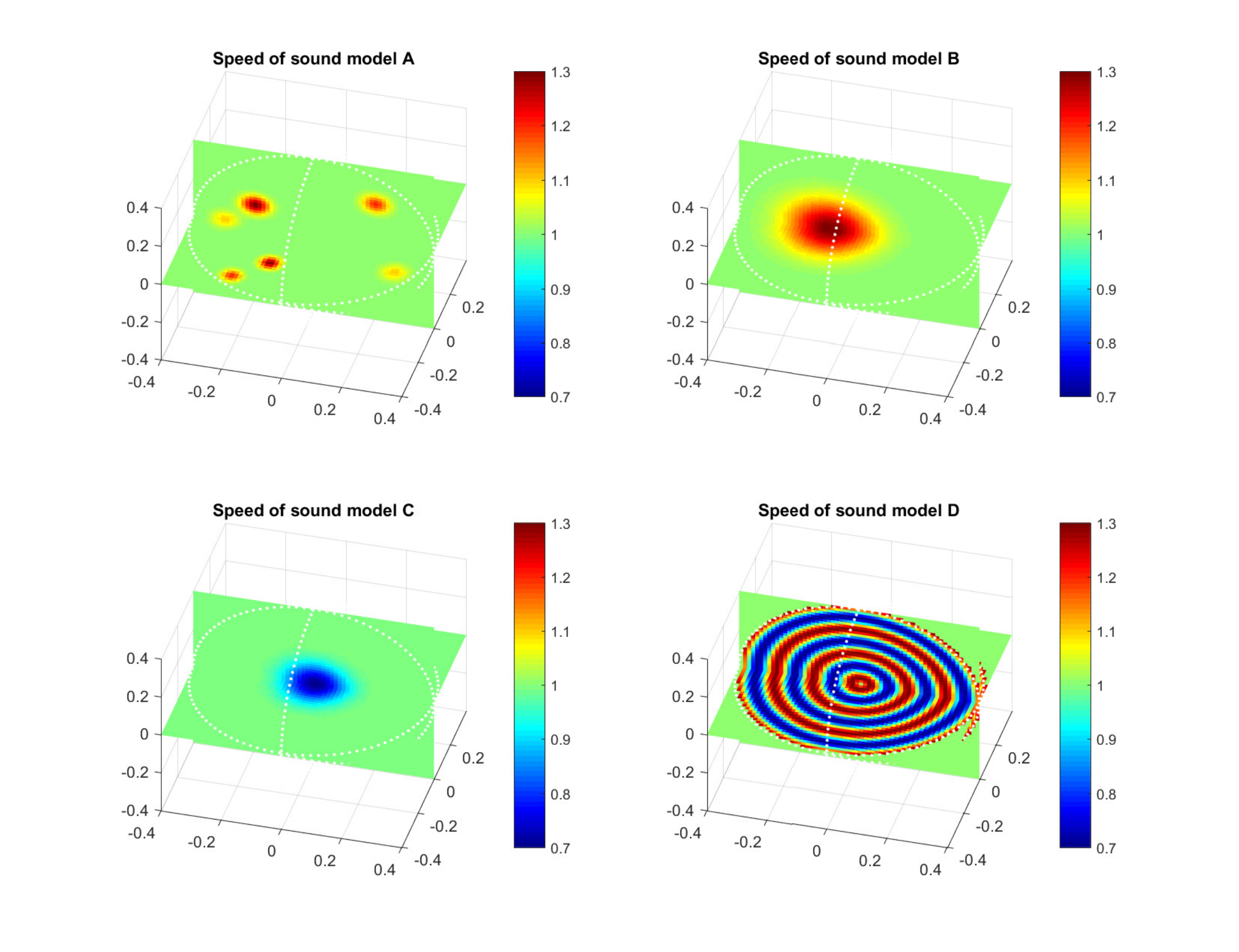}
	\caption{Top Row: Non-trapping variable speed of sound cases A and B.
	Bottom  row: Trapping speed of sound cases C and D. All cases are assumed to be distortions of the constant background speed of sound $c_0 = 1$.}\label{fig:sos}
\end{figure}

\subsection{Sound speed models and phantom}

In our numerical setup,  we use four different variable sound speed models (A, B, C and D)
which are  shown in Figure~\ref{fig:sos}.
\begin{figure}
	\centering
    \includegraphics[width=0.6\textwidth]{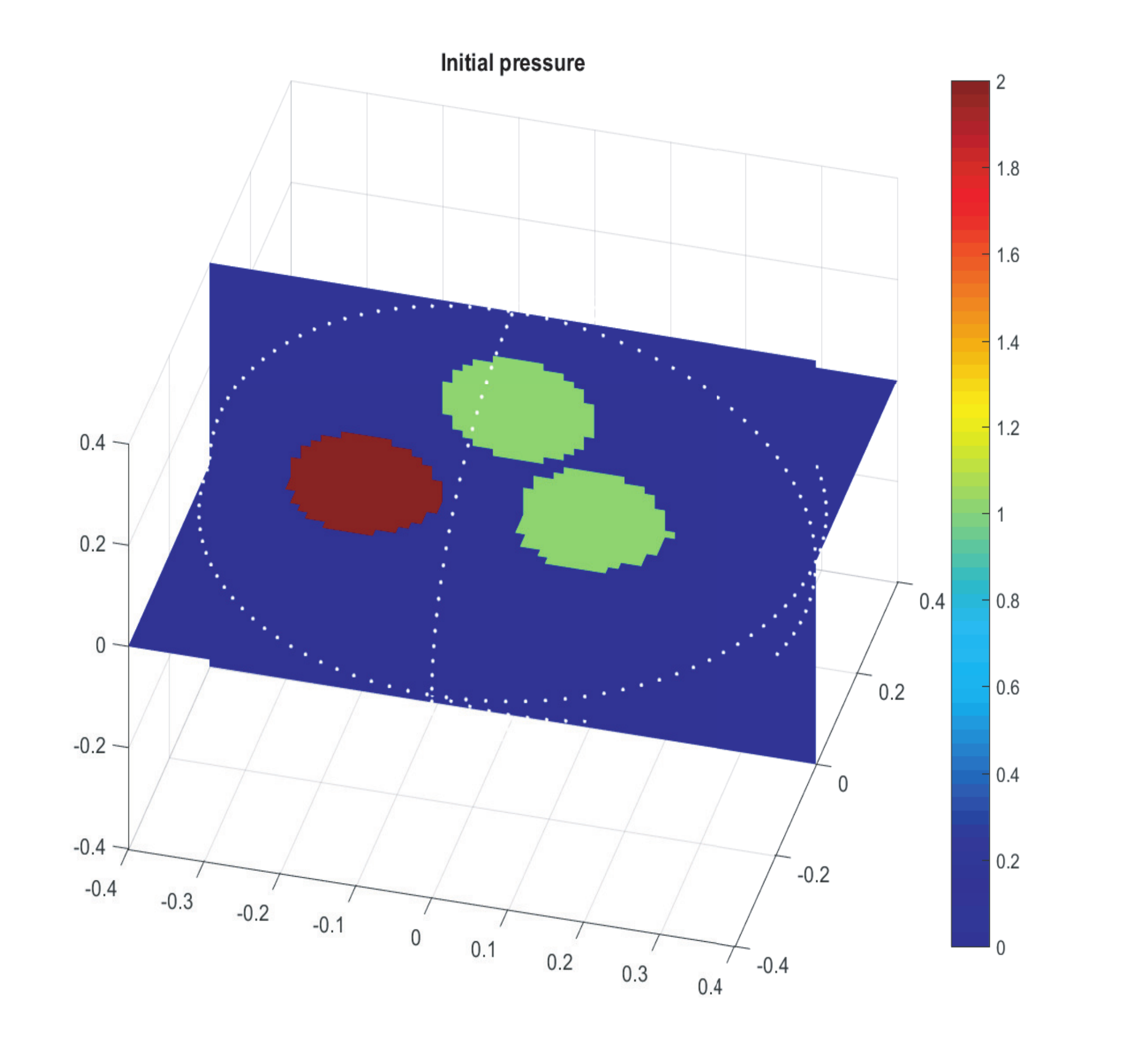}
	\caption{Initial pressure $f$ contained in  $B_a$ with $a = 0.4$ used for the presented numerical results.}\label{fig:Initialp}
\end{figure}
All variable sound speed models deviate within $30\%$ from the background sound speed $c_0 = 1$.
The two speed of sound models A and B (shown in the upper row) are acoustically non-trapping whereas
the speed of sound models C and D (shown  the bottom row) are trapping.
The speed of sound models A and B have the form
\begin{alignat}{2}
c(\f x)  = 1 +   \chi_{B_a}(\f x)  \sum_{j=1}^m  e^{- \alpha_m |\f x - \f y_j|^2},
\end{alignat}
where a sum of Gaussian pulses centered at $\f y_j$ added to the background sound speed.
The first non-trapping speed of sound model A consists of several pulses with small width, whereas the second model B
is a single pulse with a very large width.

In the trapping case C we consider a cavity in the middle of region $B_a$, which is the difference of the constant
speed of sound with a Gaussian pulse. The sound speed $D$ of sound is of the type
\begin{align*}
c(\f x)  =
\begin{cases}
 1 +   \beta \sin\left( \alpha |\f x|^2 \right) \quad 	&\f x \in B_a   \\
 1                         								&\f x  \not \in B_a\,.
\end{cases}
\end{align*}
Since in  cases C and D, $c$ is radially symmetric circles concentric to the origin are closed geodesics that never
leave $B_a$. Therefore, this sound speed cases serve as an test case for a trapping speed of sound.

We assume  the sum  of three solid spheres as initial pressure $f$,
which is depicted in Figure~\ref{fig:Initialp}.
The initial pressure distribution is contained inside the $B_a$ of radius $a = 0.4$.
In all the experiments we choose $T = 1.4$ and take $b = 2$.

\begin{figure}[htb!]
	\centering
	\includegraphics[width=\textwidth]{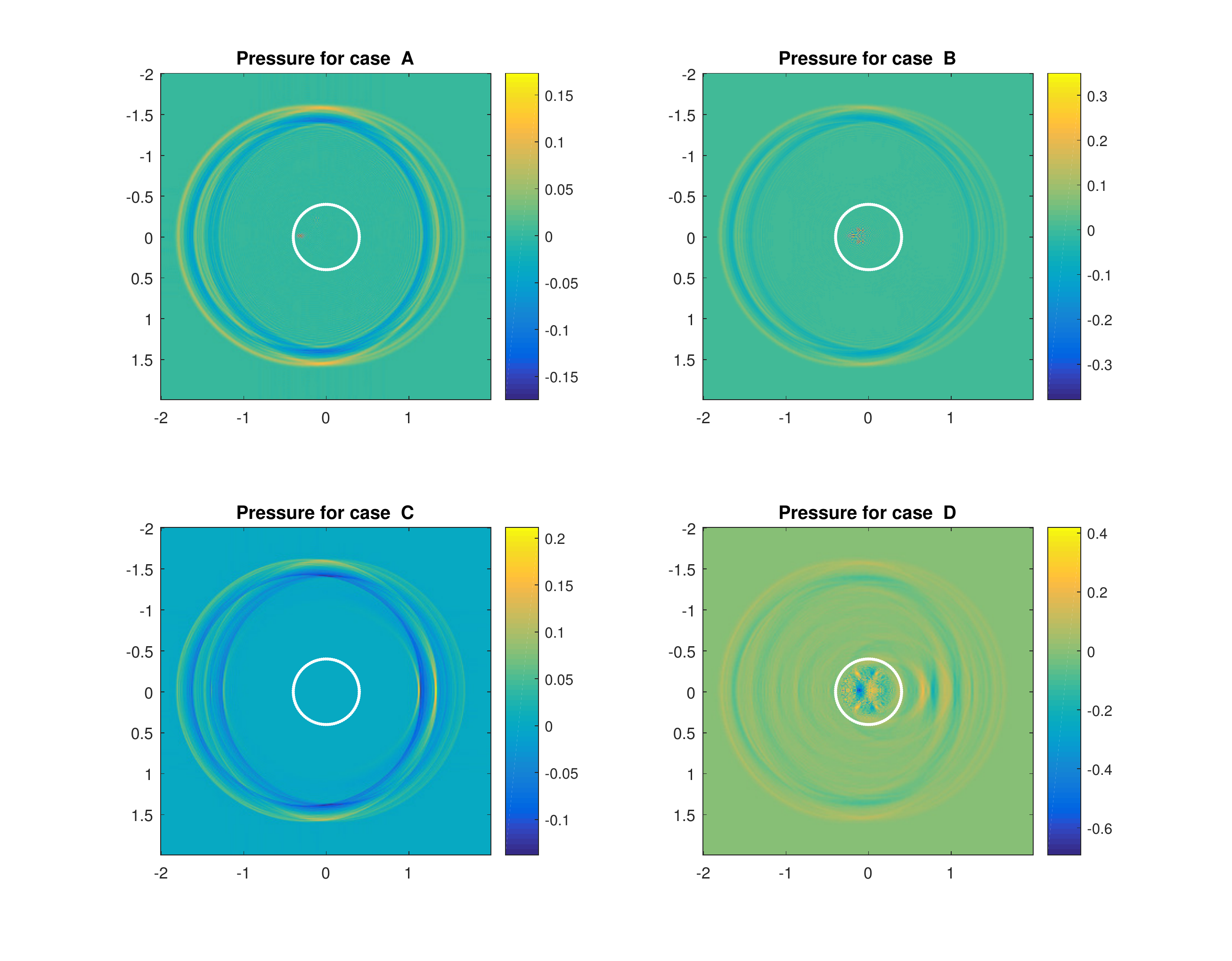}
    \caption{Slice of $\Ao f = p(\edot, T)$ through the plane $z=0$ for the
    sound speed  models A-D.}\label{fig:SlicesVS}
\end{figure}

\begin{figure}[htb!]
	\centering
    \includegraphics[width=\textwidth]{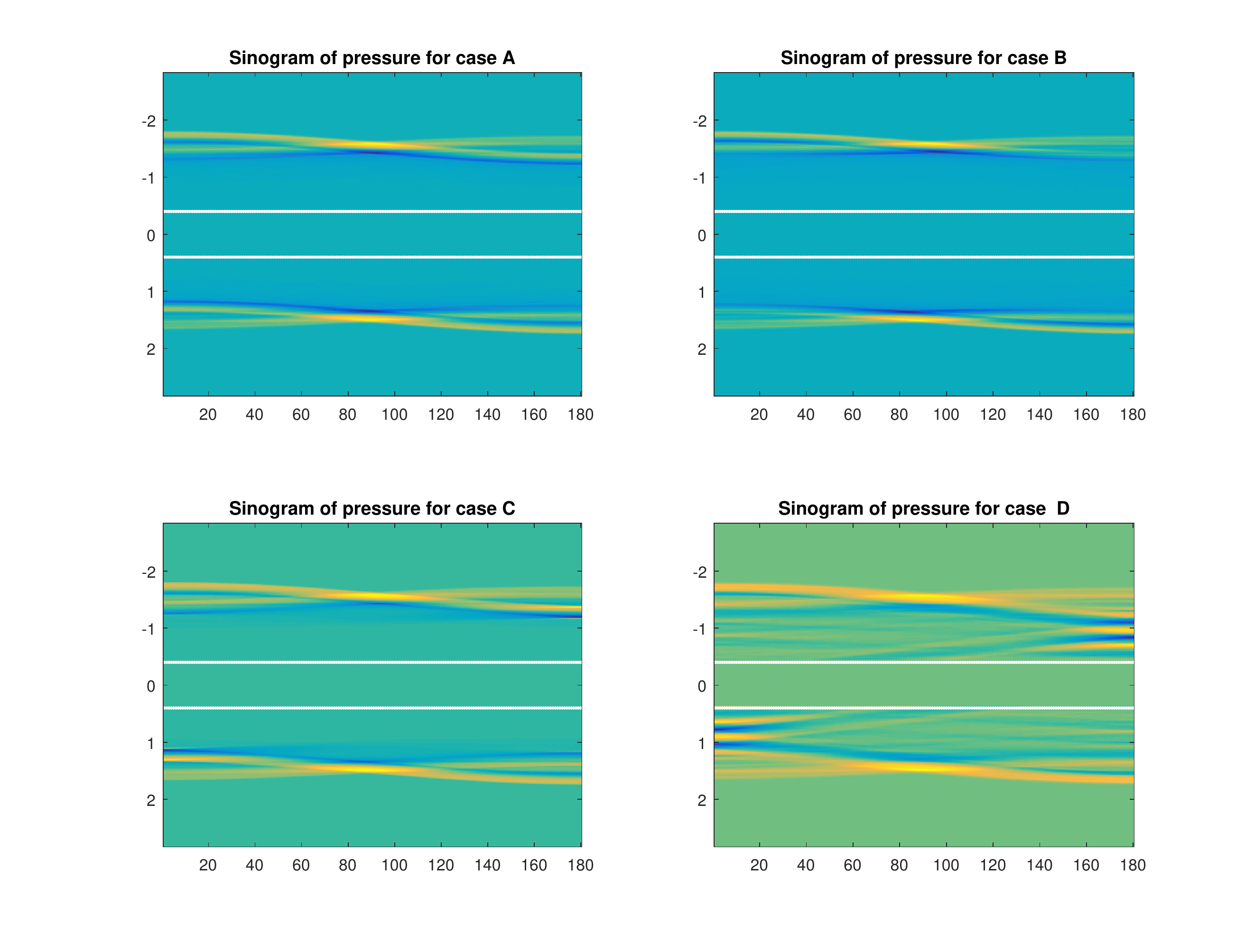}
	\caption{Simulated data $\Rad_a \Ao f$ at $z=0$ for sound speed models A-D.
	Values between the white lines (determined by the set $M_a$) extended are zero.}\label{fig:Radon}
\end{figure}

\begin{figure}[h!]
	\centering
    \includegraphics[width=\textwidth]{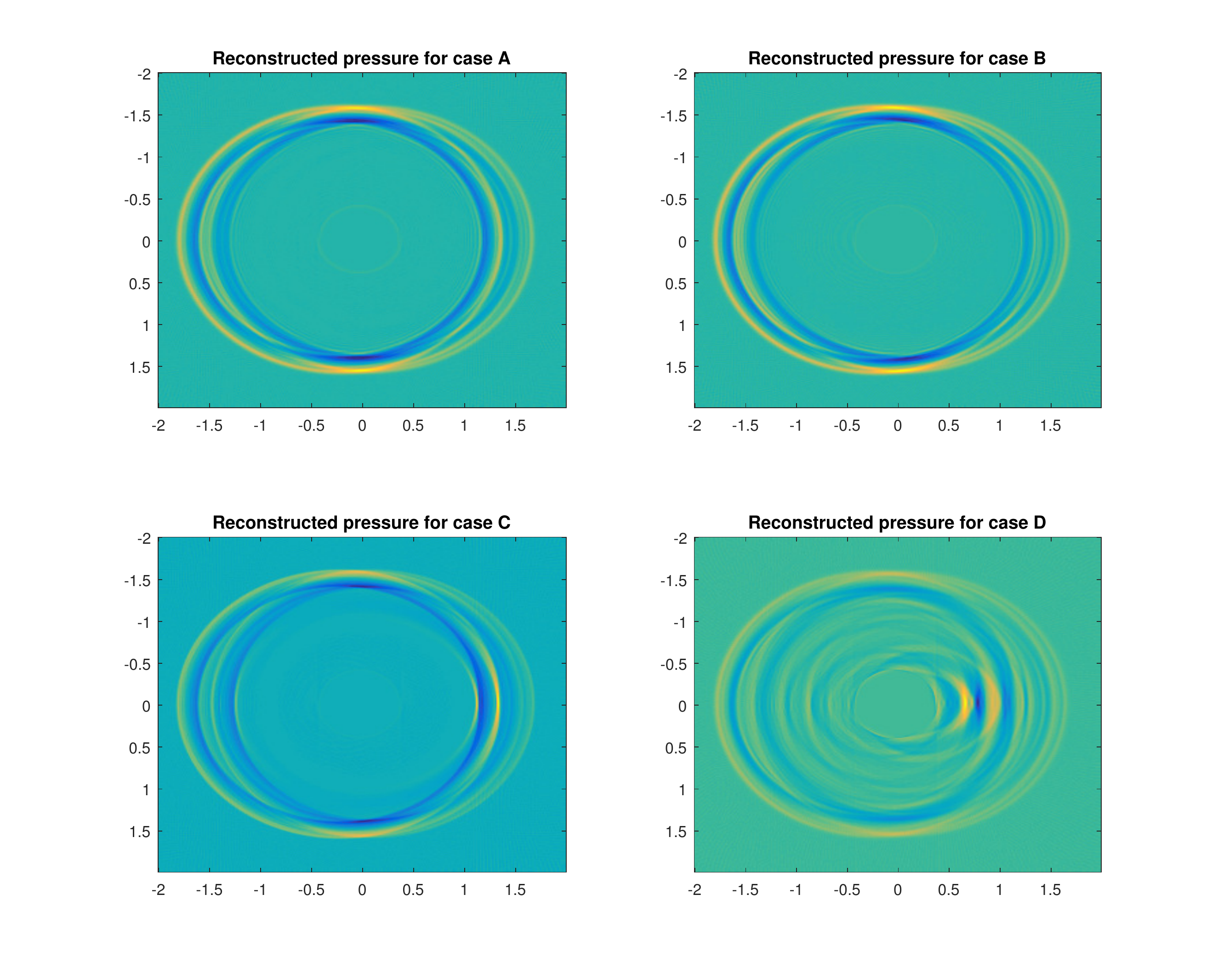}
	\caption{Slices of  the reconstruction of $\Ao f$ through the plane $z=0$ by
	applying $ \Rad^\sharp$ for the speed of sound models A-D.}\label{fig:IRadon}
\end{figure}

\subsection{Pressure simulation}

Figure~\ref{fig:SlicesVS} shows a slice of the numerically simulated data pressure
at $z=0$ for  the different speed of sound models A-D.
The simulations show that in the trapping speed of sound case pressure also decreases inside $B_a$
but at a slower rate.
Figure \ref{fig:Radon} shows full field data $g_a = \Rad_a p(\edot, T)$ for $z = 0$ and the different speed of sound models,
which are the exterior Radon transform of $p(\edot, T)$ restricted to the $xy$-plane. In the first three pictures, we see
that the Radon transform almost vanishes in the complement of $M_a$.

\begin{figure}[htb!]
	\centering
	\includegraphics[width=\textwidth]{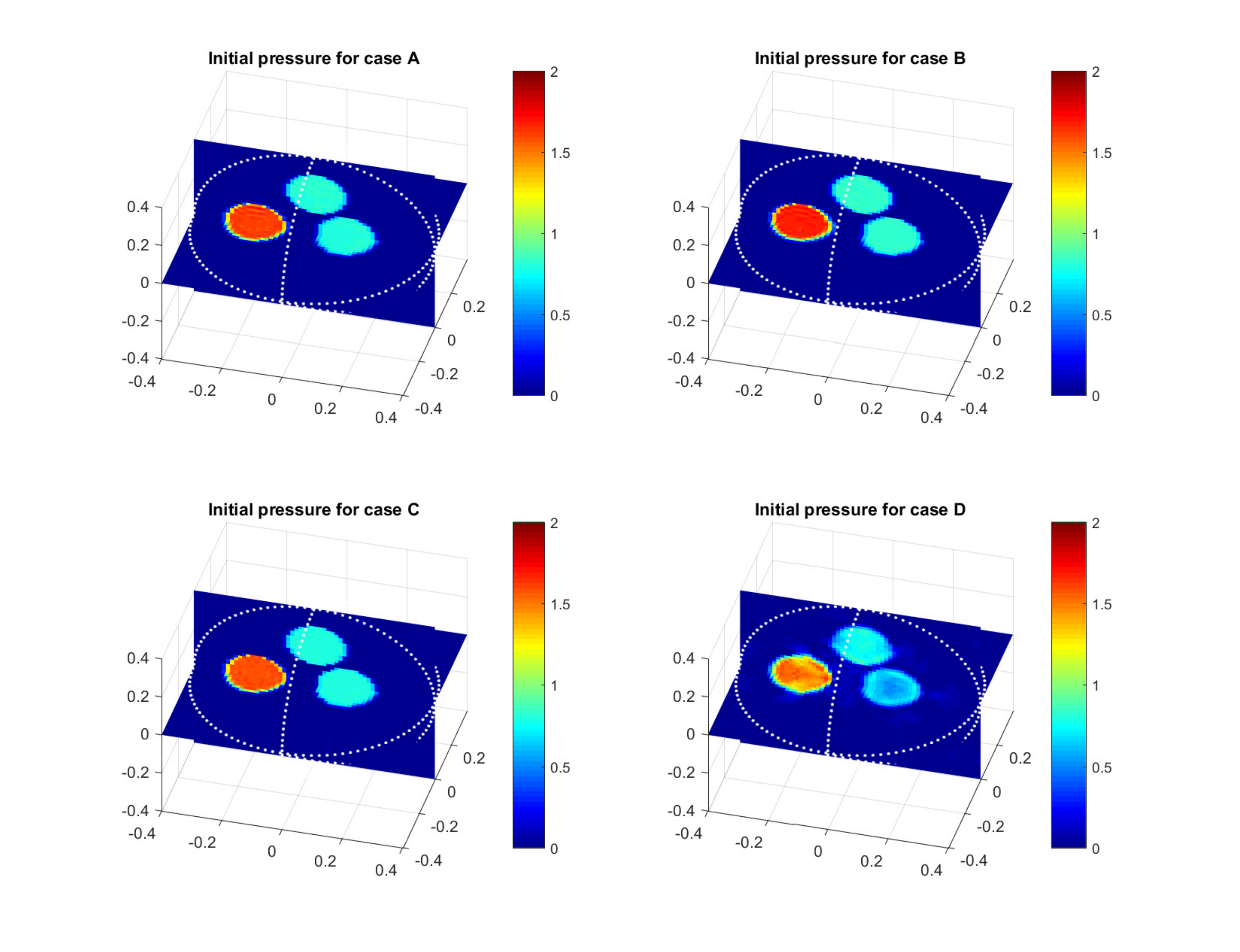}
	\caption{Reconstructions of the initial after four iterations with the  CG algorithm
	 for the speed of sound models A-D.}\label{fig:Initialp4}
\end{figure}

\begin{figure}[htb!]
	\centering
	\includegraphics[width=\textwidth]{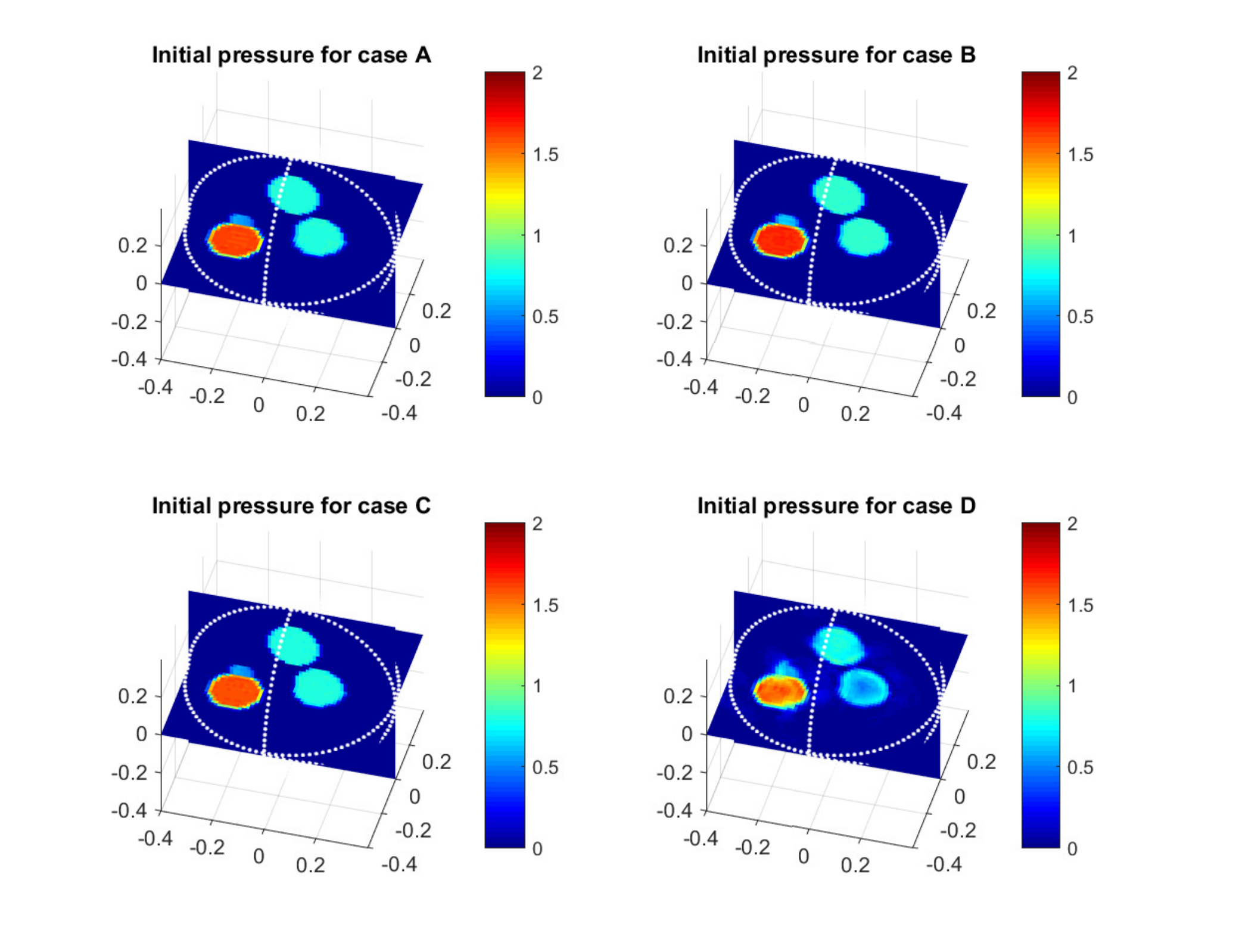}
	\caption{Reconstructions of the initial after one iteration with the  CG algorithm
	 for the speed of sound models A-D.}\label{fig:Initialp1}
\end{figure}

\subsection{Application of the CG method}\label{sex:num:Iterativ}

Figure \ref{fig:IRadon} shows reconstructions of the pressure fields in the $xy-$plane from the sinograms depicted Figure~\ref{fig:Radon}. As expected the pressure fields are reconstructed very accurately in the case of a non-trapping speed of sound.

Once point-wise data $p(\edot, T)$ are approximated on $\R^3$ we use the CG method outlined in section \ref{sec:MainResults} to numerically compute the initial pressure $f$ of system \eqref{eq:wave1}--\eqref{eq:wave3}.  Figure~\ref{fig:Initialp4} shows slice images of the initial pressure corresponding to the different sound speed cases after four iterations of the CG method.

As expected reconstructions are better for the non-trapping speed of sound models. Figure \ref{fig:Initialp1}
shows the output of the algorithm after one iteration. Finally, Figure \ref{fig:InitialwrongSOS} shows the reconstruction result,
when a wrong speed of sound (namely constant value 1) is used.  From this we can clearly see that not
accounting for  variable speed can introduce a significant error.

\begin{figure}[htb!]
	\centering
	\includegraphics[width=0.7\textwidth]{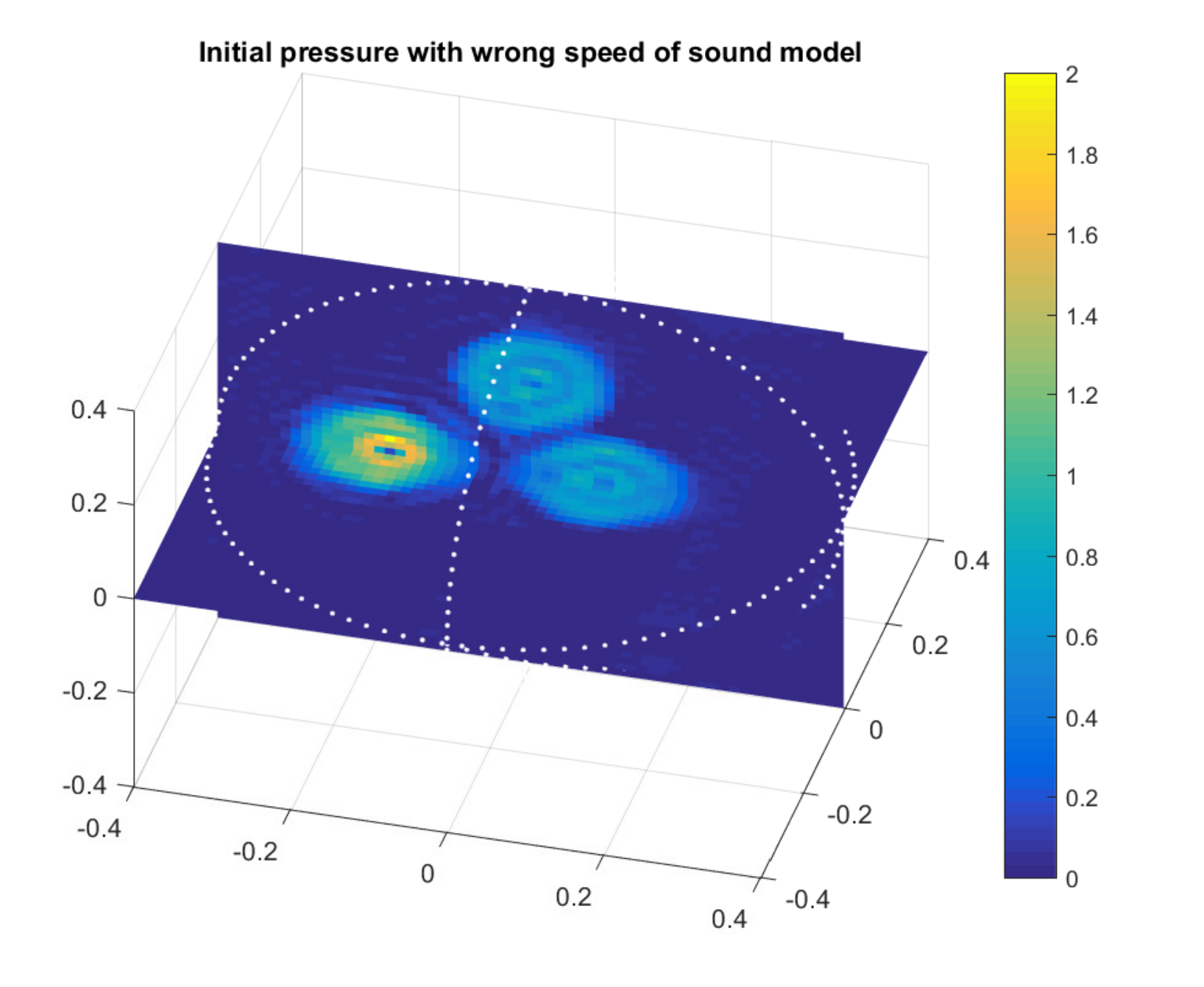}
	\caption{Reconstruction of the initial pressure after four iterations, when data are generated with speed of sound model B but the iterative CG algorithm assumes constant speed of sound with value one.}\label{fig:InitialwrongSOS}
\end{figure}

\section{Conclusion}

In this paper, we described a FFD-PAT method, where projection data of acoustic pressure are measured. For the first time, we consider this method for variable speed of sound. We developed a two-step reconstruction procedure that computes 3D acoustic pressure data point-wise in a first step and then uses them as input for an iterative algorithm in a second step. We prove uniqueness and stability estimates for the second step.  Furthermore, in upcoming works we will also study additional limited view problems for FFD-PAT which naturally arise in applications, for instance in the case of breast imaging.

\section*{Acknowledgement}

G.Z. and M.H.  acknowledge support of the Austrian Science Fund (FWF), project P 30747-N32.
The research of L.N. is supported by the National
Science Foundation (NSF) Grants DMS 1212125 and DMS 1616904.
The work of R.N. has been supported by the FWF, project P 28032.

\appendix

\section{Appendix}

\subsection{Uniqueness and stability}
\label{app:well}

Let us prove Theorem~\ref{T:stab}. To that end, we first prove two crucial results.

\begin{proposition} \label{P:unique} Assume that $p(\edot,T) =0$, then $f \equiv 0$. That is, $\Ao$ is injective. \end{proposition}

\begin{proof} Let us construct a solution $\bar p$ of the wave equation which is periodic in time with period $4T$ such that $\bar p=p$ on $\R^3 \times [0,T]$. Once this is done, we obtain $f = \bar p (\edot,0) = \bar p (\edot,4^nT)$ for any $n$. Using Lemma~\ref{lem:decay}, we arrive at
$$f(\f x) = \lim_{n \to \infty} \bar p(\f x, 4^n T) = 0, \quad \forall~ \f x \in \R^3.$$
	
It now remains to construct the above-mentioned solution $\bar p$ of the wave equation. The idea is to properly reflect the solution $p$ in the time variable $t$ through the time moments $t=T,2T, \dots,$ as follows. We first construct $\bar p$ on $[0,2T]$ by the odd reflection of $p$ through the moment $t=T$: $\bar p(\edot,T) = p(\edot,T)$ for $t \in [0,T]$ and $\bar p(\edot,T) = - p(\cdot,2T-t)$ for all $t \in [T,2T]$. Since $p(\edot,T)=0$ on $\R^3$, we obtain that $\bar p$ and $\bar p_t$ are continuous at $t=T$. Therefore, $p$ is continuous on $[0,2T]$ and solves the wave equation on that interval.
	
	We note that $\bar p_t(\cdot,2T)=-\bar p_t(\cdot,0) =0$ on $\R^3$. By the even reflection through $t=2T$: $\bar p(\edot,T) = \bar p(\edot, 4T- t)$ for all $t \in [2T,4T]$, we obtain that $\bar p$ is a solution of the wave equation in $[0, 4T]$.  Finally, we extend the solution by periodicity with period $4T$. Noting that $\bar p(\edot, 0) = \bar p(\edot,4T)$ and $\bar p_t(\edot,0) = \bar p_t(\edot,4T) =0$, we obtain that $\bar p$ and $\bar p_t$ are continuous for all time and $\bar p$ satisfies the wave equation in $\R^3 \times \R_+$. This finishes our proof.
\end{proof}

\begin{proposition} \label{P:estimate}  There is a constant $C$ such that $$\|f\|_{L^2(\Omega)} \leq 2(\|\Ao f\|_{L^2(\R^3)} + \|Kf\|_{L^2(B_a)}),$$ where $K$ is a pseudo-differential operator of order at most $-1$.
\end{proposition}

\begin{proof} Let us recall the parametrix formula for the solution $p$ of the wave equation (\ref{eq:wave1})--(\ref{eq:wave3}) (e.g., \cite{Treves2}):
	\begin{eqnarray*}
	p(\f y,t) &=& \frac{1}{(2 \pi)^3} \sum_{\pm} \int_{\R^3}  a_\pm(\f y,t,\xi) e^{i \phi_\pm(\f y,T,\xi)} \hat f(\xi) \, d\xi = \sum_{\pm} p_\pm(\f y,t).
	\end{eqnarray*}
	Here, the phase function $\phi_\pm$ solves the eikonal equation
	$$\partial_t \phi_\pm(\f y,t,\xi)  \pm  c(\f y) \, |\nabla_{\f y} \phi_\pm(\f y,t,\xi)| = 0 , \quad (\f y,t) \in \R^3 \times \R_+ $$ with the initial condition
	$$\phi_\pm(\f x,0,\xi) = \f x \cdot \xi. $$
	The amplitude function is a classical symbol $a_\pm(\f y,t,\xi) =\sum_{k=0}^\infty a_{-k,\pm}(\f y,t,\xi)$, where $a_{-k}$ is homogeneous of order $-k$ in $\xi$. Its leading term $a_{0,\pm}$ satisfies the transport equation
\begin{equation} \label{E:transport} \big(\partial_t \phi_\pm(\f y,t,\xi)  \partial_t   - c^2(\f y)  \nabla_{\f y} \phi_\pm(\f y,t,\xi) \cdot \nabla_{\f y} + C_\pm(\f y, t,\xi) \big) a_{0,\pm}(\f y,t,\xi) = 0,\end{equation} with the initial condition $a_{\pm,0} (\f x,0,\xi) = \frac{1}{2}$ (see \cite{SteUhl09}). Here, $C(\f y, \xi, t)$ only depends on the sound speed $c$ and the phase function $\phi_\pm$. Let us denote by $\gamma_{\f x,\xi}$ the unit speed geodesics originating at $\f x$ along the direction $\xi$. Then, $\gamma_{\f x, \xi}$ is a characteristics curve of the above transport equation; that is, (\ref{E:transport}) reduces to a homogeneous ODE on each geodesic curve.

We then write
\begin{eqnarray*} \Ao(f)(\f y) &=& \frac{1}{(2 \pi)^3} \sum_{\pm} \int_{\R^3}  a_\pm(\f y,T,\xi) e^{i \phi_\pm(\f y,T,\xi)} \hat f(\xi) \, d\xi  = \sum_{\pm} \Ao_\pm(f)(\f y).
	\end{eqnarray*}	
	Each operator $\Ao_\pm$ is a Fourier integral operator (FIO) with the canonical relation given by the pairs $(\f y_\pm,\lambda \eta_\pm; \f x,\lambda \xi)$ for any $\lambda \in \R$, $\xi,\eta$ unit vectors, $\f y_\pm=\gamma_{\f x,\xi}(\pm T)$, and $\eta_\pm=\dot \gamma_{\f x, \xi}( \pm T)$. Let $\R^3$ be equipped with the metrics $c^{-2}(\f x)~ d \f x^2$. Then, $(\f y_\pm,\eta_\pm)$ is obtained by translating $(\f x,\xi)$ on the geodesic $\gamma_{\f x, \pm \xi}$ by the distance $T$.  From the initial condition of $\phi_\pm$ and $a_{0,\pm}$ we see that, up to lower order terms, \begin{equation}\label{E:split} p_-(\f x,0)= p_+(\f x,0)= \frac{1}{2} f(\f x). \end{equation} Heuristically, under equation (\ref{eq:wave1})--(\ref{eq:wave3}), each singularity of $f$ at $(\f x,\xi)$ is broken into two equal parts. They propagate along the geodesic $\gamma_{\f x, \xi}$ in the opposite directions $\pm \xi$ to generate a singularity of $\Ao(f)$ at $(\f y_\pm, \eta_\pm)$.
	
	From the standard theory of FIOs (see \cite{Hormander}), the adjoint $\Ao_\pm^*$ translates $(\f y_\pm, \eta_\pm)$ back to $(\f x, \xi)$ and $\Ao^*_\pm \Ao_\pm$ is a pseudo differential operator. On the other hand, $\Ao_\mp^* \Ao_\pm$ is a FIO whose canonical relation consists of the pairs $(\f y, \eta; \f x, \xi)$ given by $\f y=\gamma_{\f x,\xi}(\pm 2T)$, and $\eta=\dot \gamma_{\f x,\xi}(\pm 2T)$. That is, $\Ao_\pm^* \Ao_\mp$ is an infinitely smoothing operator on $B$. Therefore, microlocally, we can write
	\begin{eqnarray}
	\Ao^* \Ao f &=& \Ao_+^* \Ao_+(f) + \Ao_-^* \Ao_- (f).
	\end{eqnarray}
We will show that the principal symbol $\alpha_\pm(\f x,\xi)$ of $\Ao_\pm^* \Ao_\pm$ satisfies $\alpha_\pm(\f x, \xi) = \frac{1}{4}$.
This result can be intuitively understood as follows.
Let us consider $\Ao_+^* \Ao_+$ and a singularity of $f$ at $(\f x,\xi)$. Under equation (\ref{eq:wave1})--(\ref{eq:wave3}), half of this singularity propagates into the direction $\xi$ (corresponding to the function $p_+$). At the moment $t = T$, it is transformed to a singularity of $\Ao_+(f) = p_+(T)$ at $(\f y_+, \eta_+)$. Under the adjoint equation (\ref{eq:dual}), half of this singularity propagates back to $(\f x, \xi)$ at $t=0$ to generate a singularity of $\Ao_+^* \Ao_+(f)$. It is natural to believe that this recovered singularity is $\frac{1}{4}$ of the original singularity of $f$ (due to twice splitting, as described). The proof below verify this intuition.

Indeed, denote by $q_+$ the solution of the time-reversed wave equation, e.g., equation \eqref{eq:dual}, with the initial condition given by $g_+ =  \Ao_+(f)$. Then,  by definition (see Theorem~\ref{thm:dual}) $ \Ao^* g_+ = q_+(\edot, 0)|_B$. The solution $q_+$ can be decomposed into the sum $q_+=q_0 +q_1$. Here, $q_0,q_1$, up to smooth terms, are solutions of the wave equations in $\R^3 \times (0,T)$ and satisfy $q_0(\edot,0) = \Ao_+^*(g_+)$, $q_1(\edot,0) = \Ao_-^* (g_+)$.  We are only concerned with $q_0$ since it defines $\Ao^*_+ \Ao_+ f = q_0(\edot, 0)$. We can write
\begin{equation}\label{E:q0}
q_0(\f y,t) = \frac{1}{(2 \pi)^3} \int_{\R^3}  b(\f y,t,\xi) e^{i \phi_+(\f y,t,\xi)} \hat f(\xi) \, d\xi .
\end{equation}
Let $b_0$ be the principal part of $b$. Then, the principal symbol $\alpha_+$ of $\Ao^*_+ \Ao_+$ is given by $\alpha_+(\f x, \xi) = b_0(\f x, 0, \xi)$. We note that $b_0$ satisfies the same equation as $a_{0,+}$ (see \eqref{E:transport} ). Therefore, on each bicharacteristic curve the ratio $b_0/a_{0,+}$ is constant. That implies $b_0(\f x, 0, \xi) = a_{+,0}(\f x, 0, \xi) \frac{b_0(\f y_+, T, \eta_+)}{a_{+,0}(\f y_+, T, \eta_+)}$.
Similarly to the argument below equation (\ref{E:split}), up to lower order terms, we have $$q_0(\f y_+,T) = \frac{1}{2} g_+(\f y_+,T)= \frac{1}{(2 \pi)^3} \int_{\R^3}   \frac{1}{2}  a_+(\f y_+,T,\xi) e^{i \phi_+(\f y_+,T,\xi)} \hat f(\xi) \, d\xi.$$
This and equation \eqref{E:q0} implies that $b_{0}(\f y_+, T,\xi) = \frac{1}{2} a_{+,0}(\f y_+, T,\xi)$. Therefore we obtain
$$b_{0} (\f x, 0,\xi) = \frac{1}{2} a_{+,0}(\f x, 0,\xi) = \frac{1}{4} .$$
Combining with a similar argument for $\Ao_-^* \Ao_-$, we obtain that the principal symbol of $\Ao^* \Ao$ is $\alpha(\f x, \xi)= \frac{1}{2}$. That is, $\Ao^* \Ao= \frac{1}{2} I + K$, where $K$ is a pseudodifferential operator of order at most $-1$.

Now $$(\Ao f, \Ao f) = (\Ao^* \Ao f, f) = \frac{1}{2} (f,f) + (K f,f).$$
We conclude that
	$$\|f\|^2_{L^2} \leq 2 (\|\Ao f\|^2_{L^2}  + \|Kf\|^2_{L^2}) \,.  \qedhere $$
\end{proof}

We are now ready to prove Theorem~\ref{T:stab}.

\begin{proof}[\bf Proof of Theorem~\ref{T:stab}]
Let us recall from Proposition~\ref{P:estimate} $$\|f\|_{L^2(B_a)} \leq 2(\|\Ao f\|_{L^2(\R^3)} + \|Kf\|_{L^2(B_a)}),$$ where $K$ is a pseudo-differential operator of order at most $-1$. Since $K$ is compact and $\Ao$ is injective, applying \cite[Theorem V.3.1]{taylor1981pseudodifferential}, we obtain
	$$\|f\|_{L^2(B_a)} \leq C \|\Ao f\|_{L^2(\R^3)} $$
for some constant $C \in (0, \infty)$.
This finishes our proof.
\end{proof}

\subsection{$k$-space method}
\label{app:kspace}

We briefly describe the $k$-space method  for the 3D wave equation \eqref{eq:wave1}--\eqref{eq:wave3}
as we use it for the numerical computation of  $\Ao$ and $\Ao^*$.
The $k$-space method is an attractive alternative  to standard methods
using finite differences, finite  elements or pseudospectral methods,
since it does not suffer from numerical dispersion \cite{cox2007,mast02}.
It utilizes the decomposition $p\left(\f x, t \right) = w\left(\f x, t \right) - v(\f x, t)$, where $v, w$ are defined by
\begin{align*}
w(\f x, t) :=  \frac{c_0^2}{c^2(\f x )} \,  p(\f x, t) \, \, \text{and} \, \, v(\f x, t) = \left(  \frac{c_0^2}{c^2(\f x )}  - 1 \right)  p(\f x, t),
\end{align*}
where $c_0 := \max \{c(\f x) : \f x \in \R^3\} $ denotes maximal speed of sound. It can be checked that with this definition of $v$ and $w$ the wave equation with variable speed of sound splits into the system
\begin{alignat*}{2}
w_{tt} (\f x, t) - c_0^2  \Delta w(\f x, t) &=   - c_0^2 \Delta v(\f x, t),\\						
v(\f x, t)		&=  \frac{c_0^2 - c(\f x)^2}{c_0^2} w(\f x, t).
\end{alignat*}

In the $k$-space method we use the time stepping formula
\begin{alignat}{2}\label{eq:TimeStep}
w(\f x, t + h_t) &= 2 w(\f x, t ) - w(\f x, t - h_t)\\*[0.2cm]
&-4 \Ft_\xi^{-1} \left\{ \sin\left(  \frac{c_0 |\xi| h_t}{2}   \right)^2   \Ft_{\f x} \left\{ w(\f x, t) - v(\f x, t)  \right\}  -  \left(\frac{c_0 h_t}{2} \right)^2       \right\}, \nonumber
\end{alignat}
where $\Ft_x$ and $\Ft_\xi^{-1}$ denote the Fourier transform and its inverse with respect to space and frequency variables $\f x$ and $\xi$ and $h_t$ is the time step size.
This equivalent formulation motivates the following algorithm for numerically
solving the wave equation.

\begin{framed}
\begin{alg}[$k$-space method for numerically solving {\eqref{eq:wave1}}--{\eqref{eq:wave3}}]\mbox{}\vspace{-1em}
	\begin{enumerate}[label=(S\arabic*),itemsep=0em]
		\item Define initial conditions:
		\begin{alignat*}{2}
		w(\f x, - h_t) = &w(\f x, 0) &&= c_0^2/c^2(\f x) f(\f x), \\
		&v(\f x, 0)     &&=  ( c_0^2/c^2(\f  x) - 1) f(\f x)
		\end{alignat*}
		\item   	Set $t =0$
		\item       Compute $w(\f x, t + h_t)$ according to equation \eqref{eq:TimeStep}
		\item  	    Compute $v(\f x, t + h_t):= (c^2(\f x)/c_0^2 -1 ) w (\f x, h_t)   $
		\item 		Compute $p(\f x, t + h_t):= w(\f x, t + h_t) - w(\f x, t + h_t)$
		\item       Substitute $ t $ by $t + h_t$ and go back to step (3).
	\end{enumerate}	
\end{alg}
\end{framed}

\end{document}